%% file: Morita_contextsArXiv1.tex
\documentclass[reqno]{amsart}
\usepackage{epsfig,latexsym,amsfonts,amssymb,amsmath,amscd}
\usepackage{amsthm}

\usepackage{tikz-cd}
\usepackage[mathscr]{eucal}
\usepackage{mathrsfs}
\usepackage{mathbbol}
\usepackage{array}
\usetikzlibrary{matrix,arrows,decorations.pathmorphing}
\usepackage{oldgerm,units,color}

\usepackage{pst-node}
\usepackage{graphicx}

\usepackage{eepic, epic}

\usepackage{ifthen}

\def\module0{module$^\dagger$}

\def\ssemiring0{$s$-semiring$^\dagger$}
\newtheorem*{nothma}{\textbf{Theorem A}}

\usepackage{tikz}
\input tropMacro.tex
\definecolor{lgray}{gray}{0.90}

\def\({\left(}
\def\){\right)}

\def\pipe{{\underset{{\ \, }}{\mid}}}

\def\vsemifield0{$\nu$-semifield$^\dagger$}
\def\vsemiring0{$\nu$-semiring$^\dagger$}

\def\pipe1{{\underset{{1}}{\mid}}}
\def\lmod1{\mathrel  \pipe1  \joinrel \joinrel =}

\def\CFunFF1{\operatorname{CFun} (F,F)}

\def\semiring0{semiring$^{\dagger}$}
\def\Semiring0{Semiring$^{\dagger}$}
\def\Semirings0{Semirings$^{\dagger}$}
\def\semidomain0{semidomain$^{\dagger}$}
\def\semifield0{semifield$^{\dagger}$}
\def\semifields0{semifields$^{\dagger}$}
\def\vsemifields0{$\nu$-semifields$^{\dagger}$}
\def\domain0{domain$^{\dagger}$}
\def\predomain0{pre-domain$^{\dagger}$}
\def\predomains0{pre-domains$^{\dagger}$}
\def\domains0{domains$^{\dagger}$}

\def\vdomains0{$\nu$-domains$^{\dagger}$}

\def\domains0{domains$^\dagger$}

\def\ker{\operatorname{ker}}

\newcommand{\etype}[1]{\renewcommand{\labelenumi}{(#1{enumi})}}
\def\eroman{\etype{\roman}}

\def\pipe{{\underset{{\tG}}{\mid}}}

\def\lmod{\mathrel  \pipe \joinrel \joinrel =}
\def\pipe{{\underset{{\tG}}{\mid}}}

\newtheorem{thm}[theorem]{Theorem}

\newtheorem*{thm*}{Theorem}

\def\Hom{\operatorname{Hom}}

\newtheorem{lem}[theorem]{Lemma}

\newtheorem{prop*}{Proposition}
\newtheorem{conj*}{Conjecture}

\newtheorem{prop}[theorem]{Proposition}
\newtheorem{defn}[theorem]{Definition}
\newtheorem*{examp*}{Example}
\newtheorem*{examples*}{Examples}
\newtheorem*{remark*}{Remark}
\newtheorem*{defn*}{Definition}

\def\tT{\mathcal T}

\def\tTz{\tT_\zero}

\numberwithin{equation}{section}

\def\M0{M_{\zero}}

\def\PS{P}

\def\Cong{\Phi}

\def\semirings0{semirings$^\dagger$}

\newcommand{\nPS}[1]{\PS_{(!#1)}}
\newcommand{\nPSo}[1]{\nPS{\one}}

\begin{document}


\title[Morita theory of systems]
{Morita theory of systems}

******************************* authors *********************************



\author[L.~Rowen]{Louis Rowen}
\address{Department of Mathematics, Bar-Ilan University, Ramat-Gan 52900,
Israel} \email{rowen@math.biu.ac.il}


\date{\today}




\begin{abstract}

We present the rudiments of the Morita theory of module systems,
paralleling the classical Morita theory over associative rings.
\end{abstract}

\maketitle





\section{Introduction}

Systems were introduced in \cite{Row16,JuR1} (and applied in
\cite{GaR,JMR,AKR}) to unify
 the algebraic theories of supertropical algebra, hyperfields, and fuzzy rings,
  as surveyed in \cite{Row17}. In \cite{JMR}
 we considered projective modules over commutative ground systems. In this modest note we
 provide a systemic version of
  Morita's theorem, following Bass' classical approach, as
given for example in    ~\cite[\S 4.1]{Row06}, and  formulated over
semirings in \cite[\S~3]{KN}. One interesting facet in using the
``surpassing relation'' $\preceq$ is a new kind of duality which
arises in Theorem~\ref{Mor1} and Proposition~\ref{Morplus}.

%
%

\begin{nothma}[Theorem~\ref{Mor1}, Proposition~\ref{Morplus} -- Morita's theorem for systems] Given  a systemic Morita context
$(\mathcal A, \mathcal A', \mathcal M, \mathcal M', \tau, \tau')$,

\begin{enumerate} \item
\begin{enumerate}\eroman \item If   $\tau'$ is $\preceq$-onto, then  $\mathcal M$ is a $\preceq$-progenerator for $\mathcal
A$-Mod.
\item  If $\tau$ is $\succeq$-onto, then  $\tau' $ is null-monic.
\end{enumerate}

\item The analogous statements hold if we switch left and right, or
$(\mathcal A,\tau,\mathcal M)$ and $(\mathcal A',\tau',\mathcal
M')$.
\end{enumerate}
\end{nothma}

Although the motivation comes from commutative semirings (tropical
algebra and hyperrings), we do not require commutativity.
\section{Basic notions}

A \textbf{semiring}  $(\mathcal A, +, \cdot, 1)$
is an additive commutative semigroup $(\mathcal A, +, \zero)$ and
multiplicative monoid $(\mathcal A, \cdot, \one)$ satisfying $\zero
b = b \zero = \zero$ for all $b \in \mathcal A$, as well as the
usual distributive laws. The semiring predominantly used in tropical
mathematics has been the max-plus algebra, where  $\oplus$
designates $\max$, and $\odot $ designates $+$. However, this
notation is cumbersome in an algebraic development, and also
conflicts with more customary algebraic uses of this notation, so we
proceed with the familiar algebraic notation of addition and
multiplication in whichever semiring is under consideration.

We review some more definitions from \cite{Row16,JMR,JMR2,JuR1} for
the reader's convenience.

\begin{defn}\label{modu2}   A (left) $\tT$-\textbf{\module0} over a set $ \tT$
 is an additive monoid $( \mathcal A,+)$ with a scalar
multiplication $\tT\times \mathcal A \to \mathcal A$ satisfying the
following axiom:

\medskip

 (Distributivity with respect to $\tT$): $a (\sum _{j=1}^u
b_j) = \sum _{j=1}^u (a b_j). $

\medskip

A  $\tT$-\textbf{module} is a  $\tT$-\module0 with a distinguished
element $\zero\in \mathcal A$  satisfying $a\zero_{\mathcal A }=
\zero_{\mathcal A }$ for all $a\in \tT$. To avoid complications, we
assume that $\tT\neq \emptyset$  and work only with $\tT$-modules in
this paper. We define bimodules in the usual way (i.e., satisfying
the classical associativity condition).

\end{defn}
%

Here $\mathcal A$ will be a module over a multiplicative monoid
$\tT$,  with extra structure. When $\mathcal A$ is a semiring, we
essentially have Lorscheid's blueprints, \cite{Lor1,Lor2}.

\subsection{Triples and systems}$ $

\begin{defn}\label{negmap}
 A \textbf{negation map} on  a $\tT$-module $\mathcal A$
is a
 semigroup isomorphism
$(-) :\mathcal A \to \mathcal A$ of order~$\le 2,$  written
$a\mapsto (-)a$, together with a map $(-)$  of order~$\le 2$ on $
\tT$ which also
 respects the $\tT$-action in the sense that
$$((-)a)b = (-)(ab) = a ((-)b)$$ for $a \in \tT,$ $b \in \mathcal A.$ \end{defn}
$(-)\zero = \zero.$ When $(-)\one \in \tT$ it is enough to check
that $(-)b = ((-)\one)b$ for $b \in \mathcal A.$ Assortments of
negation maps are given in \cite{Row16,JuR1,GaR}.
We write $ b_1 (-)b_2 $ for $b_1+ ((-)b_2)$, and $ b ^\circ$ for $ b
(-)a$, called a \textbf{quasi-zero}. Note that  $ (-) a ^\circ=
((-)a)+a =   a ^\circ.$

 The set $\mathcal A
^\circ$ of quasi-zeros is a $\tT$-submodule of $\mathcal A $ that
plays an important role. When $\mathcal A$ is a semiring,  $\mathcal
 A^\circ$ is an ideal.

 \begin{defn}
 A   \textbf{triple} $(\mathcal A, \tT, (-))$ is a $\tT$-module
$\mathcal A$, with $\tT$  a distinguished subset of~$\mathcal A$,
called the set of \textbf{tangible elements}, and a negation map
$(-)$ on $\mathcal A$ satisfying $(-)\tT = \tT,$  in which $\tT \cap
\mathcal A^ \circ = \emptyset$ and $\tT \cup \{\zero\}$ generates $(
\mathcal A,+).$ We write~$\tTz$ for $\tT \cup \{\zero\}.$
\end{defn}

When a given $\tT$-module $\mathcal A$ does not come equipped with a
negation map, there are several ways of providing one: Either take
$(-)$ to be the identity, as is done in supertropical algebra
\cite{IR}, or we ``symmetrize''~$\mathcal A$  as in \cite{JMR}
below. In the supertropical setting, $(-)$ is the identity map. In
the hypergroup applications, $(-)$ is induced from the
hypernegation, as explained in \cite{JMR}.

%
%
\begin{defn}\label{ho}
 By a \textbf{homomorphism} $f: (\mathcal M, \tT_{\mathcal M}, (-),
\preceq)\to (\mathcal M', \tT_{\mathcal M'} , (-)', \preceq')$ of
module triples we mean in the usual universal algebraic
 sense, i.e.,  for $a  \in \tT$ and   $b$, $b_i$  in
$\mathcal M$:
\begin{enumerate}\eroman
\item
$f (\zero ) = \zero  .$
\item $ f((-)b_1)=   (-)
f(b_1);$
\item $ f(b_1 + b_2) \preceq ' f(b_1) + f( b_2) ;$
\item  $ f(a  b)=  a   f( b) $.
\end{enumerate}
\end{defn}

 A \textbf{homomorphism} of semiring triples is also
required to satisfy $ f(b_0 b_1)= f(b_0 ) f( b_1) .$

 We round out the structure  with a \textbf{surpassing relation}
$\preceq$ given in \cite[Definition~1.70]{Row16} and also  described
in  \cite[Definition~3.11]{JuR1}.

\begin{defn}\label{precedeq07}
A \textbf{surpassing relation} on a triple $(\mathcal A, \tT, (-))$,
denoted
  $\preceq $, is a partial pre-order satisfying the following, for elements $a, a_i \in \tT$ and $b_i \in \mathcal A$:

  \begin{enumerate}
 \eroman
    \item  $\zero \preceq c^\circ$
 for any $c\in \mathcal A$.
\item  If $b_1 \preceq b_2$ then $(-)b_1 \preceq (-)b_2$.
  \item If $b_1 \preceq b_2$ and $b_1' \preceq b_2'$ for $i= 1,2$ then  $b_1 + b_1' \preceq b_2
   + b_2'.$
    \item   If  $a \in \tT$ and $b_1 \preceq b_2$ then $a b_1 \preceq ab_2.$
    \item   If  $a_0\preceq a_1 $, then $a_0=  a_1.$
   \end{enumerate}


A \textbf{$\tT$-surpassing relation} on a triple $\mathcal A$ is a
surpassing relation   satisfying the following stronger version of
(v):

If $b  \preceq a $ for $a \in \tT$ and $b \in \mathcal A$, then
$b=a$.
\end{defn}
%
%
%
%
%
%
%
%
%
%
%

\begin{example} [{\cite[Definition~1.70]{Row16},
\cite[Definition~2.17]{JuR1}}]\label{precmain}$ $
\begin{enumerate}\eroman\item  Given a triple $(\mathcal A, \tT ,
(-))$, define $a \preceq_\circ c$ if $a + b^\circ = c$ for some $b
\in \mathcal A.$ Here the surpassing relation $\preceq$ is
$\preceq_\circ$.
\item The symmetrized triple \cite[Example~1.40]{AKR} is a special case of
(i).
\item
Take $\preceq$ to be set inclusion when $\mathcal{A}$ is obtained
from a hyperfield,  see \cite[\S 10]{JuR1}.
\end{enumerate}\end{example}

Our general rule is that $\zero$ in the classical theory is replaced
by $\{ b: b \succeq \zero\},$ which is $\mathcal A^\circ$ when
$\preceq$ is $\preceq_\circ$.

%


 \begin{defn}\label{Tsyst} A \textbf{system}  is a quadruple $(\mathcal A, \tT_{\mathcal A} ,
(-), \preceq),$ where $\preceq$ is a surpassing relation  on the
triple
   $(\mathcal A, \tT_{\mathcal A}  , (-))$,
  which
 is \textbf{uniquely negated} in the sense that
for any $a \in \tT_\mathcal A$,  there is a unique element $b$ of
$\tT_\mathcal A$ for which $\zero \preceq a+b$ (namely $b =
(-)a$)\footnote{This slightly strengthens the version of ``uniquely
negated,'' for triples, used in \cite{Row16}, which says that there
is a unique element $b$ of $\tT$ for which $ a+b \in \mathcal
A^\circ.$}.

A $\tT$-\textbf{system} is a system for which $\preceq$ is a
$\tT$-surpassing relation.

\end{defn}

  There are two ways that we want to view triples and their systems. The first is as the
  ground structure on which we build our representation theory. We call this a \textbf{ground system}.
 A
range of examples of ground systems is given in
\cite[Example~2.16]{JMR}, including ``supertropical
  mathematics.''

  The second way,  initiated in \cite{JuR1},  is to fix a
   ground triple $(\mathcal A, \tT , (-))$, and  take $\mathcal A$-modules $\mathcal M$ together with
$\tT_{\mathcal M}$ satisfying $\tT \tT_{\mathcal M}\subseteq
\tT_{\mathcal M}.$  We also require the triple $(\mathcal M,
\tT_{\mathcal M}, (-))$ over a triple $(\mathcal A, \tT , (-))$ to
satisfy $((-)a)m = (-)(am)$ for $a \in \mathcal A, \ m \in \mathcal
M.$
  We call this a
  \textbf{module system} $\mathcal M : = (\mathcal M,
\tT_{\mathcal M}, (-), \preceq)$ (over $(\mathcal A, \tT , (-))$).

\subsection{$\preceq$-morphisms}\label{morph}$ $

We work over a ground system $\mathcal A= (\mathcal A, \tT_{\mathcal
A}, (-), \preceq)$ and consider $\tT$-module systems $(\mathcal M,
\tT_{\mathcal M}, (-), \preceq),$ to which we often refer merely as
$\mathcal M$.  Recall   \cite[Definition~2.37]{JuR1} that a
\textbf{$\preceq$-morphism} $$f: (\mathcal M, \tT_{\mathcal M}, (-),
\preceq)\to (\mathcal M', \tT_{\mathcal M'} , (-)', \preceq')$$  of
module  systems is a homomorphism $f: \mathcal M \to \mathcal M'$
satisfying
 the conditions of Definition~\ref{ho}, except that (iii) is weakened to $$ f(b_1 + b_2) \preceq ' f(b_1) + f( b_2)
 .$$

\begin{defn}  $ $
 \begin{enumerate}\eroman
        \item
 For any $\preceq$-morphism $f:
(\mathcal M, \tT_{\mathcal M}, (-), \preceq)  \to (\mathcal M',
\tT_{\mathcal M'}, (-), \preceq)$, define the
$\preceq$-\textbf{image}
$$f(\mathcal M )_\preceq :=  \{ b'\in  \mathcal M' : b' \preceq f(b)\quad \text{for some} \quad b\in  \mathcal M \}.$$
%

\item $f:
(\mathcal M, \tT_{\mathcal M}, (-), \preceq)  \to (\mathcal M',
\tT_{\mathcal M'}, (-), \preceq)$ is $\preceq $-\textbf{onto} if $f(
{\mathcal M} )_\preceq \, =\mathcal M',$  i.e., for   every $b \in
{\mathcal M'}$ there exists $a \in  {\mathcal M}$, for which $b
\preceq_f f(a ) $.

\item A $\preceq$-morphism $f$ is \textbf{null-monic}
  when it satisfies the property that
 if $f(b) \succeq \zero$ then $b \succeq
\zero.$
%
 \end{enumerate}
 \end{defn}

\subsection{Negated tensor products}$ $

Also we need the tensor product of   systems over a ground
$\tT$-system. These are described (for semirings) in terms of
congruences, as given for example in~\cite[Definition~3]{Ka2} or, in
our notation, \cite[\S3]{Ka3}. We do it for systems, taking the
negation map into account.

 Let
us work with a right $\mathcal A$-module system $\mathcal M_1$ and
left $\mathcal A$-module system $\mathcal M_2$ over a given ground
$\tT$-system $\mathcal A$. One  defines the tensor product $\mathcal
M_1 \otimes _{\mathcal A} \mathcal M_2$ of $\mathcal M_1$ and
$\mathcal M_2$ in the usual way, to be $(\mathcal F_1 \oplus
\mathcal F_2)/\Cong,$ where $\mathcal F_i$ is the free system
(respectively right or left) with base~$\mathcal M_i$ (and
$\tT_{\mathcal F_i} = \mathcal M_i$), and $\Cong$ is the congruence
generated by all pairs
 \begin{equation}\label{defcong}\left(\left(\sum_j x_{1,j}, \sum_k x_{2,k}\right), \sum_{j,k}\big( x_{1,j},  x_{2,k}\big)\right) ,\quad \bigg((  x_1 a, x_2 ), (x_1,a
 x_2
)\bigg),\qquad \bigg((x_1,x_2),((-)x_1,(-)x_2)\bigg)\end{equation}$
\forall x_{i,j}\in \mathcal M_i, \, a \in \mathcal A$, as well as
the extra axiom
$$((-)x) \otimes y = x \otimes ((-)y).$$
 We   define a
negation map on $\mathcal M_1 \otimes
 _{\mathcal A} \mathcal M_2$ by $(-)(v \otimes w) = ((-)
v) \otimes w.$
%
%

 \begin{defn}\label{tens1} The \textbf{negated tensor product triple} $\mathcal M_1  \otimes _{\mathcal A}\mathcal M_2$ of a right module triple  $(\mathcal M_1, \tT_1,(-))$
 and  a left module triple
$\mathcal M_2$ is $((\mathcal F_1 \oplus \mathcal F_2)/\tT_1 \times
\tT_2,\tT_{\mathcal M_1  \otimes _{\mathcal A}\mathcal M_2},(-))$,
where $\mathcal F_i$ is the free system with base~$\mathcal M_i$,
and $\tT_{\mathcal M_1  \otimes _{\mathcal A}\mathcal M_2}$ is the
set of ``simple tensors'' $a_1 \otimes a_2$ for $a_i \in \tT_i$.
\end{defn}

Any $ \mathcal A$-bilinear map $\Psi: \mathcal M_1  \times \mathcal
M_2\to  \mathcal N$ induces a map $\mathcal M_1  \otimes _{\mathcal
A}\mathcal M_2 \to  \mathcal N$ sending $x_1 \otimes x_2 \mapsto
\Psi(x_1 ,x_2).$

\section{The systemic Morita theory}$ $

As indicated in the introduction, we are interested here in a Morita
theory, with the objective of pushing Bass' methods as far as we
can.

\subsection{The systemic Morita Theorem}\label{proj1}$ $

Bass' approach to Morita's Theorem does not use negation, so can be
formulated over semirings, as done in \cite[\S~3]{KN}. We do it here
for systems and $\preceq$-morphisms, in order to handle more cases,
including the hyperfield case.

\subsubsection{Projective and $\preceq$-projective modules}\label{proj}$ $

  Projective  modules over semirings have been treated in
\cite{Tak}, \cite[Chapter~17]{golan92},   \cite{Mac} and
\cite{IKR6}. We  view them here as  module systems over ground
$\tT$-systems, but often
focus on the modules themselves. 

\begin{defn}\label{precproj} 
 (\cite{DeP,Ka3,KN,Tak}) A module system $\mathcal P: =   (\mathcal P,\tT,(-),\preceq)$ is   \textbf{projective} if for any  onto morphism of  module systems $h: \mathcal M \to \mathcal
  M'$,
 every  morphism $ f: \mathcal P \to \mathcal M'$ \textbf{lifts} to a
 morphism $\tilde f: \mathcal P \to\mathcal M$,
 in the sense that $h\tilde f =  f.$

(\cite{JMR}) $\mathcal P$ is  $\preceq$-\textbf{projective} if for
any  $\preceq$-onto morphism  $h: \mathcal M \to \mathcal M',$
 every $\preceq$-morphism $ f: \mathcal P \to \mathcal M'$ $\preceq$-\textbf{lifts} to a
 $\preceq$-morphism $\tilde f: \mathcal P \to \mathcal M$,
 in the sense that $f \preceq h\tilde f  .$
\end{defn}

The basic properties of $\preceq$-projective module systems are
given in \cite[\S4]{JMR}, including equivalent conditions
\cite[Proposition~4.3]{JMR} and the dual basis lemma
\cite[Proposition~4.16]{JMR}, leading to Schanuel's lemma
\cite[Theorem~4.21]{JMR}, $\preceq$-projective resolutions, and
$\preceq$-projective dimension \cite[Theorem~4.26]{JMR}.

\subsubsection{Generators}$ $

 \cite[Definition~3.8]{KN} defines $\mathcal N$ to be a
 \textbf{generator} if for every $\mathcal A$-module $\mathcal M$
there is an index set $I$ and an onto homomorphism $\mathcal
N^{(I)}\to \mathcal M$.

\begin{defn}\label{tr0} An $\mathcal A$-module triple
$ (\mathcal N, \tT_{\mathcal N},(-))$  is a \textbf{generator}  of a
triple $\mathcal A: = (\mathcal A, \tT,(-),\preceq)$ if for every
$\mathcal A$-module triple $ (\mathcal M, \tT_{\mathcal M},(-),
\preceq)$ there is an onto homomorphism $\mathcal N^{(I)}\to
\mathcal M$ for a suitable  index set $I$.

 An $\mathcal A$-module system $  (\mathcal N, \tT_{\mathcal N},(-), \preceq)$  is a
$\preceq$-\textbf{generator}  of a system $\mathcal A: = (\mathcal
A, \tT,(-), \preceq)$ if for every $\mathcal A$-module system $
(\mathcal M, \tT_{\mathcal M},(-), \preceq)$ there is an index set
$I$ and a $\preceq$-onto $\preceq$-morphism $\mathcal N^{(I)}\to
\mathcal M$.
\end{defn}

Any generator over $\mathcal A$ clearly is a $\preceq$-generator.
For example, $\mathcal A$ itself is a generator and a
$\preceq$-generator.

 Given an $\mathcal A$-module $\mathcal M,$ we define
$\mathcal M^* = \Hom (\mathcal M,\mathcal A),$ the semigroup of
homomorphisms from $\mathcal M$ to~$\mathcal A$, and $\mathcal
M^*_\preceq  $ to be the semigroup of $\preceq$-morphisms from
$\mathcal M$ to $\mathcal A$.

\begin{defn}\label{tr1}
The \textbf{trace ideal} $T(A)$ is $\{\sum_{\operatorname{finite}}
f(a): f \in \mathcal M^*, \ a \in \mathcal A \}.$

The $\preceq  $-\textbf{trace ideal} $T(A)_\preceq  $ is
$\{\sum_{\operatorname{finite}} f(a): f \in \mathcal M^*_\preceq, \
a \in \mathcal A \}.$

  An element  $b \in \mathcal A$ is
 \textbf{$\preceq$-generated} by $\{ a_i : 1 \le i \le t\} \subset \mathcal A$ if $b \preceq \sum _i a_i.$
\end{defn}

By  {\cite[Proposition~3.9]{KN}}, the following conditions are
equivalent:
 \begin{enumerate}\eroman
   \item $\mathcal M$ is a generator of the system $\mathcal A$.
\item $T(\mathcal M)$ generates ${\mathcal A }$.
\item There exists $n$ such that some homomorphic image of $\mathcal M^{(n)}$ generates ${\mathcal A }$. \end{enumerate}

\begin{lem}\label{tr3}
The following conditions are equivalent:
 \begin{enumerate}\eroman
   \item $\mathcal M$ is a $\preceq$-generator.
\item $T(\mathcal M)$ $\preceq$-generates ${\mathcal A }$.
\item There exists $n$ such that some $\preceq$-morphic image of $\mathcal M^{(n)}$ $\preceq$-generates ${\mathcal A }$. \end{enumerate}

\end{lem}
\begin{proof} Analogous; say, following  \cite[Lemma 4.1.7]{Row06}:

$(i) \Rightarrow (ii)$ If $h: \mathcal M^{(I)} \to \mathcal A$ is
$\preceq$-onto then $\one  \preceq h((a_i))$ for some $(a_i)$ in
$\mathcal M^{(I)}$. Taking the canonical injections $\mu_i: \mathcal
M \to \mathcal M^{(I)}$   we define $f_i = h\mu _i : \mathcal M \to
\mathcal A,$ and have $\one \preceq \sum _i f_i(a_i) \in T(\mathcal
M).$

$(ii) \Rightarrow (iii)$ The sum  $\one \preceq \sum _i f_i(a_i)$ is
finite.

$(iii) \Rightarrow (i)$ The property of  $\preceq$-generation is
transitive.
\end{proof}

\begin{defn}
A $\preceq$-\textbf{progenerator} is a  $\preceq$-finitely generated
$\preceq$-projective module which is a $\preceq$-generator.
\end{defn}

\begin{defn}\label{tr2}
A \textbf{(systemic) Morita context} is a six-tuple $(\mathcal A,
\mathcal A', \mathcal M, \mathcal M', \tau, \tau')$ where $\mathcal
A, \mathcal A'$ are systems, $\mathcal M $ is an $\mathcal A -
\mathcal A'$ bimodule,  $\mathcal M' $ is an $\mathcal A' - \mathcal
A$ bimodule, and $$\tau: \mathcal M  \otimes _{\mathcal A'} \mathcal
M' \to \mathcal A, \qquad \tau': \mathcal M'  \otimes _{\mathcal A}
\mathcal M \to \mathcal A'$$ are  homomorphisms, linear on each side
over $ {\mathcal A}$ and $ {\mathcal A'}$ respectively, which
satisfy the following equations, writing $(x,x')$ for $\tau(x,x')$
and $[x',x]$ for $\tau(x',x)$:
 \begin{enumerate}\eroman
   \item $(x,x')y =x[x',y].$
\item $ x'(x,y')=[x',x]y'.$
  \end{enumerate}
 \end{defn}
\begin{lem}
Another way of describing a Morita context is to say that $\left(
\begin{matrix}
        \mathcal A& \mathcal M \\
        \mathcal M' &  \mathcal A'
        \end{matrix}\right)$
               is a semiring, whose tangible elements have tangible components
               and whose negation map and surpassing relation are given componentwise.
\end{lem}
\begin{proof} The semiring matrix multiplication provides
 bilinear maps $\mathcal M
\times  \mathcal M' \to \mathcal A $ and $\mathcal M' \times
\mathcal M  \to \mathcal A' $ which yield $\tau $ and $\tau '$
respectively.
\end{proof}

\begin{thm}\label{Mor1}
 Suppose $(\mathcal A, \mathcal A', \mathcal M, \mathcal
M', \tau, \tau')$ is a systemic Morita context. If   $\tau'$ is
$\preceq$-onto, then  $\mathcal M$ is a $\preceq$-progenerator for
$\mathcal A$-Mod.
\end{thm}
\begin{proof} Write  $[y,\phantom{y}]$ for the homomorphism $y \mapsto [y,y']$. We prove the assertion as in  \cite[page 473]{Row06}. $\mathcal M$~is a  $\preceq$-generator by Lemma~\ref{tr3}. To
prove  projectivity, we assume that  $\sum [y_j, y_j'] \succeq
\one$.
 Write $f_j =  (\phantom{x}, y_j').$ We claim that $\{(y_j,f_j):
1 \le j \le t \}$ comprise a $\preceq$-dual basis. Indeed, $$x
\preceq  \sum x[y_j',y_j] = \sum (x,y_j')y_j = \sum f_j(x) y_j.$$
\end{proof}
%

\begin{prop}\label{Morplus}
 If $\tau$ is $\succeq$-onto, then  $\tau' $ is null-monic.
\end{prop}
\begin{proof} Write $(\phantom{x}, x')$ for the homomorphism $x \mapsto (x,x')$.
We are given   $\sum  (x_i, x_i') \preceq \one$. We  claim that if
$b = \sum z_k \otimes z_k' \in \ker_N \tau$, then $b \succeq \zero.$
Indeed,
$$\sum z_k \otimes z_k' \succeq  \sum _k z_k \otimes z_k'\sum_i(x_i,x_i')
=  \sum _k  z_k \otimes \sum _i [z_k',x_i]x_i' = \sum _{i,k}z_k
[z_k',x_i] \otimes x_i' =\sum _{i,k}(z_k ,z_k')x_i \otimes x_i'
\succeq \zero.$$\end{proof}

By symmetry of construction, the analogous statements hold if we
switch  $(\tau,\mathcal M,\mathcal A)$ and $(\tau',\mathcal
M',\mathcal A')$.

\end{document}

%% file: tropMacro.tex
\usepackage{epsfig,latexsym,amsfonts,amssymb,amsmath,amscd,graphics,epic}
\usepackage{amsfonts,amssymb,amsmath,amscd,amsthm}
\usepackage[mathscr]{eucal}
\usepackage{mathrsfs}

\usepackage{mathbbol}

\usepackage{oldgerm,units}
\usepackage{wrapfig,epsfig}

\usepackage{ifthen}
\usepackage{mathbbol}

\usepackage{amsthm}
\usepackage[mathscr]{eucal}
\usepackage{mathrsfs}
\usepackage{mathbbol}
\usepackage{oldgerm,units}
\usepackage{wrapfig}

\newtheorem{theorem}{Theorem}[section]

\newtheorem{example}[theorem]{Example}

\newcommand{\one}{\mathbb{1}}
\newcommand{\zero}{\mathbb{0}}

\newcommand{\trop}[1]{\mathcal{#1}}

\newcommand{\tG}{\trop{G}}

\newcommand{\tT}{\trop{T}}

\newcommand{\Hom}{Hom}

    \ifx\proof\undefined
    \newenvironment{proof}{
    \smallskip
    \noindent\emph{Proof.}}{\hfill\(\Box\)
    \bigskip
    } \fi

\newcommand{\ifdef}[3]{\ifthenelse{\equal{#1}{true}}{#2}{#3}}